\documentclass{amsart}

 \newtheorem{thm}{Theorem}[section]
 \newtheorem{cor}[thm]{Corollary}
 
 \newtheorem{prop}[thm]{Proposition}
 \theoremstyle{definition}
 
 \theoremstyle{remark}
 \newtheorem{rem}[thm]{Remark}
\newcommand{\M}{\mathbb{M}}
\newcommand{\A}{\mathbb{A}}

\newcommand{\X}{\mathbb{X}}

\newcommand{\PH}{\phi:\mathbb{A}\times \A \rightarrow \mathbb{X}}
\begin{document}

\title[Characterizations of left derivable maps...]
 {Characterizations of left derivable maps at non-trivial
idempotents on nest algebras}

\author{ Hoger Ghahramani }
\thanks{{\scriptsize
\hskip -0.4 true cm \emph{MSC(2010)}:  47B47; 47L35.
}}

\address{Department of
Mathematics, University of Kurdistan, P. O. Box 416, Sanandaj,
Iran.}

\email{h.ghahramani@uok.ac.ir; hoger.ghahramani@yahoo.com}

\address{}

\email{}

\thanks{}

\thanks{}



\date{}

\dedicatory{}

\commby{}


\begin{abstract}
Let $Alg \mathcal{N}$ be a nest algebra associated with the nest
$ \mathcal{N}$ on a (real or complex) Banach space $\X$. Suppose
that there exists a non-trivial idempotent $P\in Alg\mathcal{N}$
with range $P(\X) \in \mathcal{N}$ and $\delta:Alg\mathcal{N}
\rightarrow Alg\mathcal{N}$ is a continuous linear mapping
(generalized) left derivable at $P$, i.e.
$\delta(ab)=a\delta(b)+b\delta(a)$
($\delta(ab)=a\delta(b)+b\delta(a)-ba\delta(I)$) for any $a,b\in
Alg\mathcal{N}$ with $ab=P$. we show that $\delta$ is a
(generalized) Jordan left derivation. Moreover, we characterize
the strongly operator topology continuous linear maps $\delta$ on
some nest algebra $Alg\mathcal{N}$ with property that
$\delta(P)=2P\delta(P)$ or $\delta(P)=2P\delta(P)-P\delta(I)$
every idempotent $P$ in $Alg\mathcal{N}$.
\end{abstract}

\maketitle

\section{Introduction}
Throughout this paper all algebras and vector spaces will be over
$\mathbb{F}$, where $\mathbb{F}$ is either the real field
$\mathbb{R}$ or the complex field $\mathbb{C}$. Let $\A$ be an
algebra with unity $1$, $\M$ be a left $\A$-module and $\delta:
\A\rightarrow \M$ be a linear mapping. $\delta$ is said to be a
\emph{left derivation} (or a \emph{generalized left derivation})
if $\delta(ab)=a\delta(b)+b\delta(a)$ (or
$\delta(ab)=a\delta(b)+b\delta(a)-ba\delta(1)$) for all $a, b\in
\A$. It is called a \emph{Jordan left derivation} (or a
\emph{generalized Jordan left derivation}) if
$\delta(a^{2})=2a\delta(a)$ (or
$\delta(a^{2})=2a\delta(a)-a^{2}\delta(1)$) for any $a \in \A$.
Obviously, any (generalized) left derivation is a (generalized)
Jordan left derivation, but in general the converse is not true
(see \cite{Za}, Example 1.1). The concepts of left derivation and
Jordan left derivation were introduced by
Bre$\check{\textrm{s}}$ar and Vukman in \cite{Bre}. For results
concerning left derivations and Jordan left derivations we refer
the readers to \cite{Gho} and the references therein.
\par
In recent years, several authors studied the linear (additive)
maps that behave like homomorphisms, derivations or left
derivations when acting on special products (for instance, see
\cite{Bre2, Gha0, Hou, Ji, Zhu} and the references therein). In
this article we study the linear maps on nest algebras behaving
like left derivations at idempotent-product elements.
\par
Let $\A$ be an algebra with unity $1$, $\M$ be a left $\A$-module
and $\delta: \A\rightarrow \M$ be a linear mapping.  We say that
$\delta$ is \emph{left derivable} (or \emph{generalized left
derivable}) at a given point $z\in \A$ if
$\delta(ab)=a\delta(b)+b\delta(a)$ (or
$\delta(ab)=a\delta(b)+b\delta(a)-ba\delta(1)$) for any $a,b \in
\A$ with $ab=z$. In this paper we characterize the continuous
linear maps on nest algebras which are (generalized) left
derivable at a non-trivial idempotent operator $P$. Moreover, we
describe the strongly operator topology continuous linear maps
$\delta$ on some nest algebra $Alg\mathcal{N}$ with property that
$\delta(P)=2P\delta(P)$ or $\delta(P)=2P\delta(P)-P\delta(I)$
every idempotent $P$ in $Alg\mathcal{N}$.
\par
The following are the notations and terminologies which are used
throughout this article.
\par
Let $\X$ be a Banach space. We denote by $\mathcal{B}(\X)$ the
algebra of all bounded linear operators on $\X$, and $
\mathcal{F}(\X)$ denotes the algebra of all finite rank operators
in $\mathcal{B}(\X)$. A \emph{subspace lattice} $\mathcal{L}$ on
a Banach space $\X$ is a collection of closed (under norm
topology) subspaces of $\X$ which is closed under the formation
of arbitrary intersection and closed linear span (denoted by
$\vee$), and which includes $\{0\}$ and $\X$. For a subspace
lattice $ \mathcal{L}$, we define $Alg\mathcal{L}$ by
\[ Alg\mathcal{L}=\{ T\in \mathcal{B}(\X)\,|\, T(N)\subseteq N
\,for\, all \,N\in \mathcal{L}\}. \]
A totally ordered subspace
lattice $ \mathcal{N}$ on $\X$ is called a \emph{nest} and
$Alg\mathcal{N}$ is called a \emph{nest algebra}. When
$\mathcal{N} \neq \{\{0\},\X\}$, we say that $\mathcal{N}$ is
non-trivial. It is clear that if $\mathcal{N}$ is trivial, then
$Alg\mathcal{N} =\mathcal{B}(\X)$. Denote
$Alg_{\mathcal{F}}\mathcal{N}:= Alg\mathcal{N} \cap
\mathcal{F}(\X)$, the set of all finite rank operators in
$Alg\mathcal{N}$ and for $N\in \mathcal{N}$, let $N_{-} = \vee
\{M\in \mathcal{N}\,|\,M \subset N \}$. The identity element of
nest algebras denote by $I$ and an element $P$ in a nest algebra
is called a \emph{non-trivial idempotent} if $P\neq 0,I$ and
$P^{2}=P$.
\par
Let $\mathcal{N}$ be a non-trivial nest on a Banach space $\X$. If
there exists a non-trivial idempotent $P\in Alg\mathcal{N}$ with
range $P(\X) \in \mathcal{N}$, then we have
$(I-P)(Alg\mathcal{N})P=\{0\}$ and hence
\[ Alg\mathcal{N}= P(Alg\mathcal{N})P\dot{+} P(Alg\mathcal{N})(I-P)\dot{+} (I-P)(Alg\mathcal{N})(I-P) \]
as sum of linear spaces. This is so-called the Peirce
decomposition of $Alg\mathcal{N}$. The sets $P(Alg\mathcal{N})P$,
$P(Alg\mathcal{N})(I-P)$ and $(I-P)(Alg\mathcal{N})(I-P)$ are
closed in $Alg\mathcal{N}$. In fact $P(Alg\mathcal{N})P$ and
$(I-P)(Alg\mathcal{N})(I-P)$ are Banach subalgebras of
$Alg\mathcal{N}$ with unity $P$ and $I-P$, respectively and
$P(Alg\mathcal{N})(I-P)$ is a Banach
$(P(Alg\mathcal{N})P,(I-P)(Alg\mathcal{N})(I-P))$-bimodule. Also
$P(Alg\mathcal{N})(I-P)$ is faithful as a left
$P(Alg\mathcal{N})P$-module as well as a right
$(I-P)(Alg\mathcal{N})(I-P))$-module. For more information on
nest algebras, we refer to \cite{Dav}.
\par
A subspace lattice $ \mathcal{L}$ on a Hilbert space $ \mathbb{H}$
is called a \emph{commutative subspace lattice}, or a \emph{CSL},
if the projections of $ \mathbb{H}$ onto the subspaces of $
\mathcal{L}$ commute with each other. If $ \mathcal{L}$ is a
$CSL$, then $Alg\mathcal{L}$ is called a \emph{CSL algebra}. Each
nest algebra on a Hilbert space is a $CSL$-algebra.
\section{Main results}
In order to prove our results we need the following result.
\begin{thm} \cite{Gha}.\label{c1}
Let $\mathbb{X}$ be a Banach space and let $\PH$ be a continuous
bilinear map with the property that
\[ a,b\in \A, \, ab=1 \Rightarrow \phi(a,b)=\phi(1,1).\]
Then
\[ \phi(a,a)=\phi(a^{2},1)\]
for all $ a \in \A$.
\end{thm}
\begin{prop}\label{ld1}
Let $\A$ be a Banach algebra with unity $1$ and $\M$ be a unital
Banach left $\A$-module. Let $\delta:\A \rightarrow \M$ be a
continuous linear map. If $\delta$ is left derivable at $1$, then
$\delta$ is a Jordan left derivation.
\end{prop}
\begin{proof}
Since $1.1=1$, it follows that $\delta(1)=2\delta(1)$. So
$\delta(1)=0$. Define a continuous bilinear map $\phi:\A \times \A
\rightarrow \M$ by $\phi(a,b)=a\delta(b)+b\delta(a)$. Then
$\phi(a,b)=\phi(1,1)$ for all $a,b \in \A$ with $ab=1$, since
$\delta$ is left derivable at $1$. By applying Theorem~\ref{c1},
we obtain $\phi(a,a)=\phi(a^{2},1)$ for all $a\in \A$. So
\[ \delta(a^{2})=2a\delta(a) \quad (a \in \A). \]
\end{proof}
\begin{cor}\label{ld2}
Let $\A$ be a Banach algebra with unity $1$ and $\M$ be a unital
Banach left $\A$-module. Let $x,y\in \A$ with $x+y=1$ and let
$\delta:\A \rightarrow \M$ be a continuous linear map. If
$\delta$ is left derivable at $x$ and $y$, then $\delta$ is a
Jordan left derivation.
\end{cor}
\begin{proof}
For $a,b\in \A$ with $ab=1$, we have $abx=x$ and $aby=y$. Since
$\delta$ is left derivable at $x$ and $y$, it follows that
\[ \delta(x)=\delta(abx)=a\delta(bx)+bx\delta(a) \]
and
\[ \delta(y)=\delta(aby)=a\delta(by)+by\delta(a). \]
Combining the two above equations, we get that
\[
\delta(1)=\delta(x+y)=a\delta(bx)+bx\delta(a)+a\delta(by)+by\delta(a)=a\delta(b)+b\delta(a),
\]
i.e. $\delta$ is left derivable at $1$. So from
Proposition~\ref{ld1}, $\delta$ is a Jordan left derivation.
\end{proof}
\begin{rem}
If $\A$ is a $CSL$-algebra or a unital semisimple Banach algebra,
then by \cite{Ji} and \cite{Vuk} every continuous Jordan left
derivation on $\A$ is zero. Hence from Proposition~\ref{ld1}
every continuous linear map $\delta:\A\rightarrow \A$ which is
left derivable at $1$ is zero.
\end{rem}
The following is our main result.
\begin{thm}\label{asli}
Let $\mathcal{N}$ be a nest on a Banach space $\X$, and there
exists a non-trivial idempotent $P\in Alg\mathcal{N}$ with range
$P(\X) \in \mathcal{N}$. If $\delta : Alg\mathcal{N}\rightarrow
Alg\mathcal{N}$ is a continuous left derivable map at $P$, then
$\delta$ is a Jordan left derivation.
\end{thm}
\begin{proof}
As a notational convenience, we denote $\A=Alg\mathcal{N}$,
$\A_{11}=P\A P$, $\A_{12}=P\A (I-P)$ and $\A_{22}=(I-P)\A (I-P)$.
As mentioned in the introduction $\A=\A_{11} \dot{+} \A_{12}
\dot{+}\A_{22}$. Throughout the proof, $a_{ij}$ and $b_{ij}$ will
denote arbitrary elements in $\A_{ij}$ for $1\leq i,j \leq 2$.
\par
First we show that $\delta(P)=0$. Since $P^{2}=P$, we have
$2P\delta(P)=\delta(P)$. So $2P\delta(P)=P\delta(P)$ and
$(I-P)\delta(P)=0$. Thus $P\delta(P)=0$ and hence $\delta(P)=0$.
\par
We complete the proof by checking some steps.
\\
\textbf{Step 1.} $P\delta(a_{11}^{2})P=2a_{11}P\delta(a_{11})P $
and $P\delta(a_{11}^{2})(I-P)=2a_{11}P\delta(a_{11})(I-P)$.
\\ \par
 For any $a_{11}, b_{11}$ with $a_{11}b_{11}=P$, we have
\begin{equation}\label{1}
 a_{11}\delta(b_{11})+b_{11}\delta(a_{11})=\delta(P).
\end{equation}
Multiplying this identity by $P$ both on
the left and on the right we find
\[ a_{11}P\delta(b_{11})P+b_{11}P\delta(a_{11})P=P\delta(P)P \quad (a_{11}b_{11}=P).\]
Define a continuous linear map $d:\A_{11}\rightarrow \A_{11}$ by
$d(a_{11})=P\delta(a_{11})P$. By above identity $d$ is left
derivable at $P$. Hence by Proposition~\ref{ld1}, $d$ is a Jordan
left derivation. So
\[ P\delta(a_{11}^{2})P=2a_{11}P\delta(a_{11})P \quad (a_{11}\in \A_{11}).\]
Now, multiplying the Equation\eqref{1} from the left by $P$, from
the right by $(I-P)$ we arrive at
\[ a_{11}P\delta(b_{11})(I-P)+b_{11}P\delta(a_{11})(I-P)=P\delta(P)(I-P) \quad (a_{11}b_{11}=P).\]
Define a continuous linear map $D:\A_{11}\rightarrow \A_{12}$ by
$D(a_{11})=P\delta(a_{11})(I-P)$. So $D$ is left derivable at $P$
and from Proposition~\ref{ld1}, $D$ is a Jordan left derivation.
Thus
\[ P\delta(a_{11}^{2})(I-P)=2a_{11}P\delta(a_{11})(I-P) \quad (a_{11}\in \A_{11}).\]
\textbf{Step 2.} $P\delta(a_{22})=0 $.
\\ \par
Since $(P+a_{22})P=P$, we have
\[ (P+a_{22})\delta(P)+P\delta(P+a_{22})=\delta(P).\]
From $\delta(P)=0$ we get $$P\delta(a_{22})=0. $$
\textbf{Step 3.} $P\delta(a_{12})=0 $.
\\ \par
Applying $\delta$ to $(P+a_{12})P=P$, we get
\[ (P+a_{12})\delta(P)+P\delta(P+a_{12})=\delta(P).\]
Since $\delta(P)=0$, it follows that $$P\delta(a_{12})=0. $$
\textbf{Step 4.} $(I-P)\delta(a_{11})(I-P)=0 $.
\\ \par
For any $a_{11}, b_{11}$ with $b_{11}a_{11}=P$, we have
$(I-P+b_{11})a_{11}=P$ and hence
\begin{equation*}
 (I-P+b_{11})\delta(a_{11})+a_{11}\delta(I-P+b_{11})=\delta(P).
\end{equation*}
Multiplying this identity by $I-P$ both on the left and on the
right we arrive at
\[ (I-P)\delta(a_{11})(I-P)=0.\]
Since any element in a Banach algebras is a sum of invertible
elements, by the linearity of $\delta$ and above identity we have
\[ (I-P)\delta(a_{11})(I-P)=0\]
for all $a_{11} \in \A_{11}$.
\\
\textbf{Step 5.} $(I-P)\delta(a_{12})(I-P)=0 $.
\\ \par
Since $(P-a_{12})(I+a_{12})=P$, it follows that
\[ (P-a_{12})\delta(I+a_{12})+(I+a_{12})\delta(P-a_{12})=\delta(P).\]
Multiplying this identity by $I-P$ both on the left and on the
right, from $\delta(P)=0$  we find
\[(I-P)\delta(a_{12})(I-P)=0.\]
\textbf{Step 6.} $(I-P)\delta(a_{22})(I-P)=0 $.
\\ \par
Applying $\delta$ to $(P+a_{12})(P-a_{12}a_{22}+a_{22})=P$, we
see that
\[ (P+a_{12})\delta(P-a_{12}a_{22}+a_{22})+(P-a_{12}a_{22}+a_{22})\delta(P+a_{12})=\delta(P).\]
Now, multiplying this identity from the left by $P$, from the
right by $I-P$ and by Steps 2,3 and 5 and the fact that
$\delta(P)=0$, we get $a_{12}(I-P)\delta(a_{22})(I-P)=0$. Since
$a_{12}\in \A_{12}$ is arbitrary, we have
$\A_{12}((I-P)\delta(a_{22})(I-P))=\{0\}$. From the fact that
$\A_{12}$ is faithful as right $\A_{22}$-module, we arrive at
$$(I-P)\delta(a_{22})(I-P)=0. $$
\par
Since $ab=PaPbP+PaPb(I-P)+Pa(I-P)b(I-P)+(I-P)a(I-P)b(I-P)$, for
any $a,b \in \A$, by Steps 1--6, it follows that $\delta$ is a
Jordan left derivation.
\end{proof}
Our next result characterizes the linear mappings on
$Alg\mathcal{N}$ which are generalized left derivable at $P$.
\begin{thm}\label{gen}
Let $\mathcal{N}$ be a nest on a Banach space $\X$, and there
exists a non-trivial idempotent $P\in Alg\mathcal{N}$ with range
$P(\X) \in \mathcal{N}$. If $\delta: Alg\mathcal{N}\rightarrow
Alg\mathcal{N}$ is a continuous generalized left derivable map at
$P$, then $\delta$ is a generalized Jordan left derivation.
\end{thm}
\begin{proof}
Define $\Delta:Alg\mathcal{N}\rightarrow Alg\mathcal{N}$ by
$\Delta(a)=\delta(a)-a\delta(1)$. It is easy too see that $\Delta$
is a continuous left derivable map at $P$. By Theorem~\ref{asli},
$\Delta$ is a Jordan left derivation. Therefore
\begin{equation*}
\begin{split}
\delta(a^{2})&=\Delta(a^{2})+a^{2}\delta(1)\\ &
=2a\Delta(a)+a^{2}\delta(1)\\ &
=2a(\delta(a)-a\delta(1))+a^{2}\delta(1)\\
&=2a\delta(a)-a^{2}\delta(1)
\end{split}
\end{equation*}
for all $a\in Alg\mathcal{N}$. So $\delta$ is a generalized Jordan
left derivation.
\end{proof}
Since every continuous Jordan left derivation on a $CSL$ algebra
is zero \cite{Ji}, we have the following result.
\begin{cor}
Let $\mathcal{N}$ be a non-trivial nest on a Hilbert space
$\mathbb{H}$. Let $P$ be a non-trivial idempotent in
$Alg\mathcal{N}$ with range $P(\mathbb{H}) \in \mathcal{N}$ and
$\delta:Alg\mathcal{N}\rightarrow Alg\mathcal{N}$ be a continuous
linear map.
\begin{enumerate}
\item[(i)] If $\delta$ is left derivable at $P$, then $\delta$ is zero.
\item[(ii)] If $\delta$ is generalized left derivable at $P$, then $\delta(a)=a\delta(1)$ for all $a\in Alg\mathcal{N}$.
\end{enumerate}
\end{cor}
\begin{proof}
(i) Since every continuous Jordan left derivation on a $CSL$
algebra is zero \cite{Ji}, by Theorem~\ref{asli}, $\delta$ is
zero.
\\
(ii) By Theorem~\ref{asli}, $\delta$ is a generalized Jordan left
derivation, so the mapping $\Delta:Alg\mathcal{N}\rightarrow
Alg\mathcal{N}$ defined by $\Delta(a)=\delta(a)-a\delta(1)$ is a
continuous Jordan left derivation. Therefore $\Delta=0$ and hence
$\delta(a)=a\delta(1)$ for all $a\in Alg\mathcal{N}$.
\end{proof}
Now, we characterize the strongly operator topology continuous
(generalized) left Jordan derivations on some nest algebras.
\begin{prop}\label{ppp}
Let $\mathcal{N}$ be a nest on a Banach space $\X$, with each
$N\in \mathcal{N}$ complemented in $\X$ whenever $N_{-} = N$. Let
$\delta: Alg\mathcal{N}\rightarrow Alg\mathcal{N}$ be a strong
operator topology continuous linear map. Then:
\begin{enumerate}
\item[(i)] If $\delta(P)=2P\delta(P)$ for every idempotent $P$ in
$Alg\mathcal{N}$, then $\delta=0$.
\item[(ii)] If $\delta(P)=2P\delta(P)-P\delta(I)$ for every idempotent $P$ in
$Alg\mathcal{N}$, then $\delta(a)=a\delta(I)$ for all $a\in
Alg\mathcal{N}$.
\end{enumerate}
\end{prop}
\begin{proof}
(i) For arbitrary idempotent operator $P\in Alg\mathcal{N}$, by
hypothesis we have $\delta(P)=2P\delta(P)$. So
$2P\delta(P)=P\delta(P)$ and $(I-P)\delta(P)=0$. Therefore
$P\delta(P)=0$ and hence $\delta(P)=0$.
\par
Notice that $Alg_{\mathcal{F}}\mathcal{N}$ is contained in the
linear span of the idempotents in $Alg\mathcal{N}$ (see
\cite{Hou}). So we see that $\delta(F)=0$ for all finite rank
operator $F$ in $Alg\mathcal{N}$. Since $\delta$ is continuous
under the strong operator topology and
$\overline{Alg_{\mathcal{F}}\mathcal{N}}^{SOT}=Alg\mathcal{N}$
(see \cite{Spa}), we find that $\delta(a)=0$ for all $a\in
Alg\mathcal{N}$.
\\
(ii) Define $\Delta:Alg\mathcal{N}\rightarrow Alg\mathcal{N}$ by
$\Delta(a)=\delta(a)-a\delta(I)$. It is easy too see that $\Delta$
is a continuous left map satisfying $\Delta(P)=2P\Delta(P)$ for
every idempotent $P$ in $Alg\mathcal{N}$. So by (i) we have
$\Delta=0$ and hence $\delta(a)=a\delta(I)$ for all $a\in
Alg\mathcal{N}$.
\end{proof}
It is obvious that the nests on Hilbert spaces, finite nests and
the nests having order-type $\omega + 1$ or $1 + \omega^{*}$,
where $\omega$ is the order-type of the natural numbers, satisfy
the condition in Proposition~\ref{ppp} automatically.



\bibliographystyle{amsplain}
\bibliography{xbib}

\end{document}